\documentclass[a4paper,10pt]{amsart}
\usepackage{amscd}
\usepackage{amssymb}
\usepackage[all]{xy}
 \usepackage{graphicx}
\usepackage[T1]{fontenc}
\usepackage{lmodern}
\date{01 May 2014}
\title[Combinatorial Descent Data]{Combinatorial Descent Data for Gerbes}
\usepackage{hyperref}
\hypersetup{colorlinks=false}

\author{Amnon Yekutieli}
\address{A. Yekutieli: Department of  Mathematics
Ben Gurion University,
Be'er Sheva 84105,
Israel}
\email{amyekut@math.bgu.ac.il}
\thanks{
Primary: 18G50; Secondary: 18G30, 20L05.}
\thanks{This research was supported by the Israel Science Foundation.}

\newtheorem{thm}[equation]{Theorem}

\newtheorem{prop}[equation]{Proposition}
\newtheorem{lem}[equation]{Lemma}
\theoremstyle{definition}
\newtheorem{dfn}[equation]{Definition}
\newtheorem{rem}[equation]{Remark}

\numberwithin{equation}{section}
\newcommand{\iso}{\xrightarrow{\simeq}}

\newcommand{\opn}{\operatorname}
\newcommand{\cat}[1]{\operatorname{\mathsf{#1}}}

\newcommand{\ol}{\overline}

\newcommand{\rmitem}[1]{\item[\text{\textup{(#1)}}]}
\newcommand{\mfrak}[1]{\mathfrak{#1}}
\newcommand{\mcal}[1]{\mathcal{#1}}

\newcommand{\mrm}[1]{\mathrm{#1}}
\newcommand{\mbb}[1]{\mathbb{#1}}

\newcommand{\tup}[1]{\textup{#1}}
\newcommand{\bsym}[1]{\boldsymbol{#1}}

\newcommand{\hatotimes}[1]{\, \what{{\otimes}}_{#1} \,}
\newcommand{\hot}{\hatotimes{}}
\newcommand{\til}[1]{\tilde{#1}}
\newcommand{\what}[1]{\widehat{#1}}

\newcommand{\K}{\mbb{K}}

\newcommand{\N}{\mbb{N}}

\newcommand{\m}{\mfrak{m}}

\newcommand{\al}{\alpha}

\newcommand{\gerbe}[1]{\bsym{\mcal{#1}}}

\newcommand{\crvar}{\curvearrowright}

\begin{document}

\begin{abstract}
We consider descent data in cosimplicial crossed groupoids. This is a
combinatorial abstraction of the descent data for gerbes in algebraic geometry. 
The main result is this: a weak equivalence between
cosimplicial crossed groupoids induces a bijection on gauge equivalence classes
of descent data. 
\end{abstract}

\maketitle

\setcounter{section}{-1}
\section{Introduction}
\label{sec:Int}

For a cosimplicial crossed groupoid
$G = \{ G^p \}_{p \in \mbb{N}}$ 
we denote by 
$\ol{\opn{Desc}}(G)$
the set of gauge equivalence classes of descent data. 
The purpose of this note is to prove: 

\begin{thm}[Equivalence] \label{thm:1} 
Let $\Phi : G \to H$ be a weak equivalence between 
cosimplicial crossed groupoids. Then the function 
\[ \ol{\opn{Desc}}(\Phi) : \ol{\opn{Desc}}(G) \to
\ol{\opn{Desc}}(H) \]
is bijective.
\end{thm}

The various notions involved are recalled or defined in Section 1. The theorem
is repeated as Theorem \ref{thm:100} in Section 2, and proved there. Connections
with other papers, and several remarks, are in Section 3. 
 
Theorem \ref{thm:1} plays a crucial role in the new version of our paper 
\cite{Ye1} on twisted deformation quantization of algebraic varieties. This is
explained in Remark \ref{rem:12}.

\medskip \noindent
\textbf{Acknowledgments.}
I wish to thank Matan Prezma, Ronald Brown, Sharon Hollander, Vladimir Hinich,
Behrang Noohi and Lawrence Breen for useful discussions.

\section{Combinatorial Descent Data}
\label{sec:cosim}

We begin with a quick review of cosimplicial theory. 
Let $\bsym{\Delta}$ denote the simplex category. The set
of objects of $\bsym{\Delta}$ is the set $\mbb{N}$ of natural
numbers. Given $p, q \in \mbb{N}$, the morphisms $\alpha : p \to q$
in $\bsym{\Delta}$ are order preserving functions
\[ \alpha : \{ 0, \ldots, p \} \to \{ 0, \ldots, q \} . \]
We denote this set of morphisms by
$\bsym{\Delta}^q_p$.
An element of $\bsym{\Delta}^q_p$ may be thought of as a sequence
$\bsym{i} = (i_0, \ldots, i_p)$ of integers with
$0 \leq i_0 \leq \cdots \leq i_p \leq q$. 
We call 
$\bsym{\Delta}^q := \{ \bsym{\Delta}^q_p \}_{p \in \mbb{N}}$
the $q$-dimensional combinatorial simplex, and an element
$\bsym{i} \in \bsym{\Delta}^q_p$ is a $p$-dimensional face of
$\bsym{\Delta}^q$.

Let $\cat{C}$ be some category. A {\em cosimplicial object} in 
$\cat{C}$ is a functor $C : \bsym{\Delta} \to \cat{C}$.
We shall usually write $C^p := C(p) \in \opn{Ob} (\cat{C})$, and leave
the morphisms
$C(\alpha) : C(p) \to C(q)$, for $\alpha \in \bsym{\Delta}^q_p$,
implicit. Thus we shall refer to the cosimplicial object $C$ as
$\{ C^p \}_{p \in \mbb{N}}$.
The category of cosimplicial objects in $\cat{C}$, where the morphisms 
are natural transformations of functors 
$\bsym{\Delta} \to \cat{C}$, is denoted by 
$\bsym{\Delta}(\cat{C})$.

If $\cat{C}$ is a category of sets with structure, then an object 
$C \in \opn{Ob}(\cat{C})$ has elements $c \in C$. 
Let $\{ C^p \}_{p \in \mbb{N}}$ be a cosimplicial object of $\cat{C}$. 
Given a face $\bsym{i} \in \bsym{\Delta}_p^q$
and an element $c \in C^p$, it will be convenient to write
\begin{equation} \label{eqn:101}
c_{\bsym{i}} := C(\bsym{i})(c) \in C^q .
\end{equation}
The picture to keep in mind is of ``the element $c$ pushed to the face
$\bsym{i}$ of the simplex $\bsym{\Delta}^q$''. 
See Figure \ref{fig:2} for an illustration.

Let $G$ be a groupoid. For objects $x, y \in \opn{Ob}(G)$ we
write $G(x,y) := \opn{Hom}_{G}(x, y)$, the set of morphisms
$g : x \to y$. We also denote by $G(x) := G(x, x)$ the automorphism group of
$x$. 

Suppose $G_1$ and $G_2$ are groupoids, such that 
$\opn{Ob}(G_1) = \opn{Ob}(G_2)$. 
An {\em action} $\Psi$ of $G_1$ on 
$G_2$ is a collection of group isomorphisms 
$\Psi(g) : G_2(x) \iso G_2(y)$
for all $x, y \in \opn{Ob}(G_1)$ and $g \in G_1(x, y)$, 
such that 
$\Psi(h \circ g) = \Psi(h) \circ \Psi(g)$
whenever $g$ and $h$ are composable, and 
$\Psi(1_x) = \bsym{1}_{G_2(x)}$. Here $1_x \in G_1(x)$ is the the identity
automorphism of the object $x$ in the groupoid $G_1$, and $\bsym{1}_{G_2(x)}$
is the identity automorphism of the group $G_2(x)$.
The prototypical example is the adjoint action $\opn{Ad}_{G_1}$ of the groupoid
$G_1$ on itself, namely 
\[ \opn{Ad}_{G_1}(g)(h) := g \circ h \circ g^{-1} . \]

\begin{dfn} \label{dfn:cosim.101}
A {\em crossed groupoid} is a structure 
\[ G = 
( G_1, G_2, 
\opn{Ad}_{G_1 \crvar G_{2}}, \opn{D} ) \]
consisting of:
\begin{itemize}
\item Groupoids $G_1$ and $G_2$, such that 
$G_2$ is totally disconnected, and 
$\opn{Ob}(G_1) = \opn{Ob}(G_2)$. 
We write $\opn{Ob}(G) := \opn{Ob}(G_1)$. 

\item An action $\opn{Ad}_{G_1 \crvar G_{2}}$ of 
$G_1$ on $G_2$, called the {\em twisting}. 

\item A morphism of groupoids (i.e.\ a functor) 
$\opn{D} : G_2 \to G_1$
called the {\em feedback}, which is the identity on objects.
\end{itemize}

These are the conditions:
\begin{enumerate}
\rmitem{i} The morphism $\opn{D}$ is $G_1$-equivariant with respect to
the actions $\opn{Ad}_{G_1 \crvar G_{2}}$ and
$\opn{Ad}_{G_1}$. Namely 
\[ \opn{D}(\opn{Ad}_{G_1 \crvar G_{2}}(g)(a)) =
\opn{Ad}_{G_1}(g)(\opn{D}(a)) \]
in the group $G_1(y)$, for any $x, y \in \opn{Ob}(G)$, 
$g \in G_1(x, y)$ and $a \in G_2(x)$.

\rmitem{ii} For any $x \in \opn{Ob}(G)$ and 
$a \in G_2(x)$ there is equality
\[ \opn{Ad}_{G_1 \crvar G_{2}}(\opn{D}(a)) =
\opn{Ad}_{G_2(x)}(a) , \]
as automorphisms of the group $G_2(x)$.
\end{enumerate}
\end{dfn}

We sometimes refer to the morphisms in the groupoid $G_1$ as {\em
$1$-morphisms}, and to the morphisms in $G_2$ as {\em $2$-morphisms}. 

\begin{rem} \label{rem:cosim.101}
A crossed groupoid is better known as a {\em strict $2$-groupoid}, or a
{\em crossed module over a groupoid}, or a {\em $2$-truncated crossed
complex}; see \cite{Bw}. When
$\opn{Ob}(G)$ is a singleton then $G$ is just a crossed
module (or a crossed group). More on this in Remark \ref{rem:11}.
\end{rem}

\begin{dfn}
Suppose 
$H = ( H_1, H_2, \opn{Ad}_{H_1 \crvar H_{2}}, \opn{D} )$ 
is another crossed groupoid. A {\em morphism of crossed groupoids}
$\Phi : G \to H$
is a pair of morphisms of groupoids 
$\Phi_i : G_i \to H_i$, $i = 1, 2$, that are equal on objects, and
respect the twistings and the feedbacks. 

We denote by
$\cat{CrGrpd}$ the category consisting of crossed groupoids and morphisms
between them.
\end{dfn}

We shall be interested in cosimplicial crossed groupoids, i.e.\ in objects of
the category $\bsym{\Delta}(\cat{CrGrpd})$. A cosimplicial crossed groupoid 
$G = \{ G^p \}_{p \in \mbb{N}}$
has a crossed groupoid $G^p$ in each simplicial dimension $p$. 
The morphisms of crossed groupoids $G(\bsym{i}) : G^p \to G^q$, 
for $\bsym{i} \in \bsym{\Delta}_p^q$, are implicit, and we use notation
(\ref{eqn:101}) for objects, $1$-morphisms and $2$-morphisms.

Let us fix $p \in \N$. Then for any 
$x \in \opn{Ob}(G^p)$ there is a group homomorphism  (the feedback)
\[ \opn{D} : G^p_2(x) \to  G^p_1(x) . \]
And for every morphism $g : x \to y$ in $G_1^p$ there is a group
isomorphism (the twisting)
\[ \opn{Ad}(g) = 
\opn{Ad}_{G^p_1 \curvearrowright G^p_{2}}(g) :
G^p_2(x) \to G^p_2(y) . \]
Note that we are using the expression $\opn{Ad}(g)$ to mean both 
$\opn{Ad}_{G^p_1 \curvearrowright G^p_{2}}(g)$
and 
$\opn{Ad}_{G^p_1}(g)$; hopefully that will not cause confusion.

\begin{dfn} \label{dfn:101}
Let $G = \{ G^p \}_{p \in \mbb{N}}$ be a cosimplicial
crossed groupoid. A {\em combinatorial descent datum} in $G$
is a triple $(x, g, a)$ of elements of the following sorts:
\begin{enumerate}
\rmitem{0}  $x \in \opn{Ob}(G^0)$. 

\rmitem{1} $g \in G^1_1(x_{(0)}, x_{(1)})$, where 
$x_{(0)}, x_{(1)} \in \opn{Ob}(G^1)$ are the objects corresponding
to the vertices $(0)$ and $(1)$ of $\bsym{\Delta}^1$.

\rmitem{2} $a \in G^2_2(x_{(0)})$, where 
$x_{(0)} \in \opn{Ob}(G^2)$ is the object corresponding
to the vertex $(0)$ of $\bsym{\Delta}^2$.
\end{enumerate}

The conditions are as follows:
\begin{enumerate}
\rmitem{i} (Failure of $1$-cocycle) 
\[ g_{(0, 2)}^{-1} \circ g_{(1, 2)} \circ g_{(0, 1)} = \opn{D}(a) \]
in the group $G^2_1(x_{(0)})$.
Here $x_{(i)} \in \opn{Ob}(G^2)$ 
and $g_{(i, j)} \in G^2_1(x_{(i)}, x_{(j)})$ 
correspond to the faces $(i)$ and $(i, j)$ respectively of $\bsym{\Delta}^2$.

\rmitem{ii} (Twisted $2$-cocycle) 
\[ a_{(0, 1, 3)}^{-1} \circ a_{(0, 2, 3)} \circ 
a_{(0, 1, 2)} = 
\opn{Ad}(g_{(0, 1)}^{-1})(a_{(1, 2, 3)}) \]
in the group
$G^3_2(x_{(0)})$.
Here $x_{(i)} \in \opn{Ob}(G^3)$, 
$g_{(i, j)} \in G^3_1(x_{(i)}, x_{(j)})$ and
$a_{(i, j, k)} \in G^3_2(x_{(i)})$
correspond to the faces $(i)$,  $(i, j)$ and $(i, j, k)$ respectively of
$\bsym{\Delta}^3$.
\end{enumerate}

We denote by $\opn{Desc}(G)$ the set of all
descent data in $G$. 
\end{dfn}

See Figure \ref{fig:2} for an illustration.

\begin{figure}
\begin{center}
\includegraphics[scale=0.55]{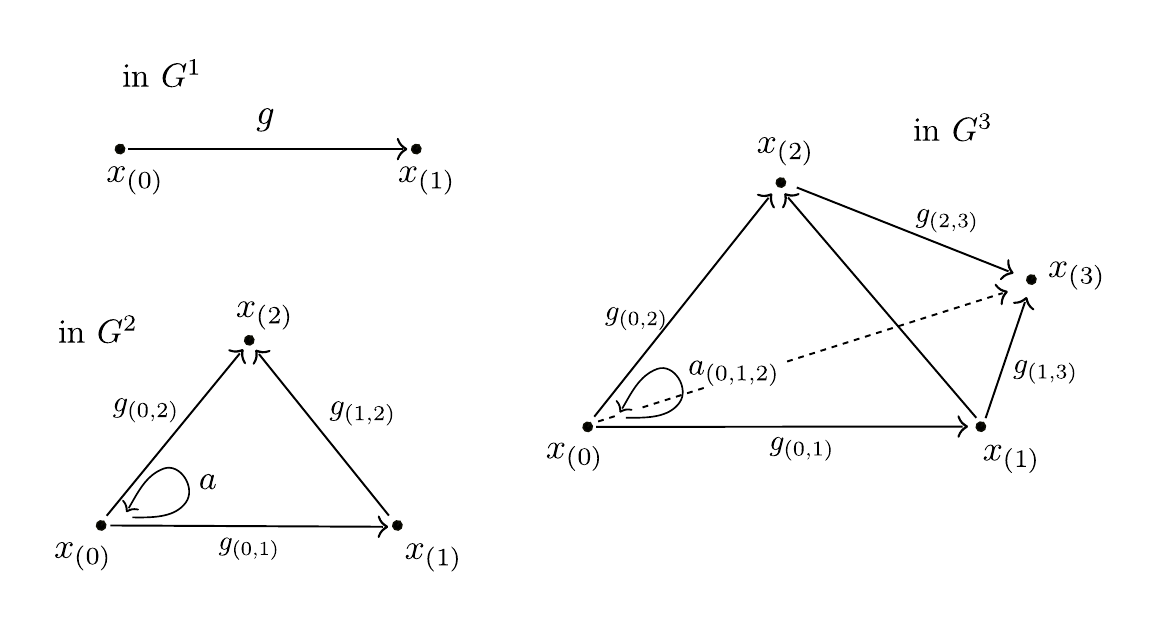}
\caption{Illustration of a combinatorial descent datum $(x, g, a)$ in the
cosimplicial crossed groupoid 
$G = \{ G^p \}_{p \in \mbb{N}}$.}
\label{fig:2}
\end{center}
\end{figure}

\begin{dfn} \label{dfn:100}
Let $(x, g, a)$ and $(x', g', a')$
be descent data in the cosimplicial crossed groupoid 
$G$. A gauge transformation
$(x, g, a) \to (x', g', a')$ is a pair $(f, c)$ 
of elements of the following sorts:
\begin{enumerate}
\rmitem{0}  $f \in G^0_1(x, x')$.

\rmitem{1} $c \in G^1_2(x_{(0)})$, where 
$x_{(0)} \in \opn{Ob}(G^1)$ is the object corresponding
to the vertex $(0)$ of $\bsym{\Delta}^1$.
\end{enumerate}

These two conditions must hold:
\begin{enumerate}
\rmitem{i}  
\[ g'  = f_{(1)} \circ g \circ \opn{D}(c) \circ f_{(0)}^{-1} \]
in the set
$G^1_1(x'_{(0)}, x'_{(1)})$.

\rmitem{ii}
\[ a' =   
\opn{Ad}(f_{(0)}) \Bigl( c_{(0, 2)}^{-1} \circ a \circ 
\opn{Ad}(g_{(0, 1)}^{-1})(c_{(1, 2)}) \circ c_{(0, 1)} \Bigr) \]
in the group $G^2_2(x'_{(0)})$.
\end{enumerate}
\end{dfn}

This is illustrated in Figure \ref{fig:3}.

\begin{figure}
\begin{center}
\includegraphics[scale=0.55]{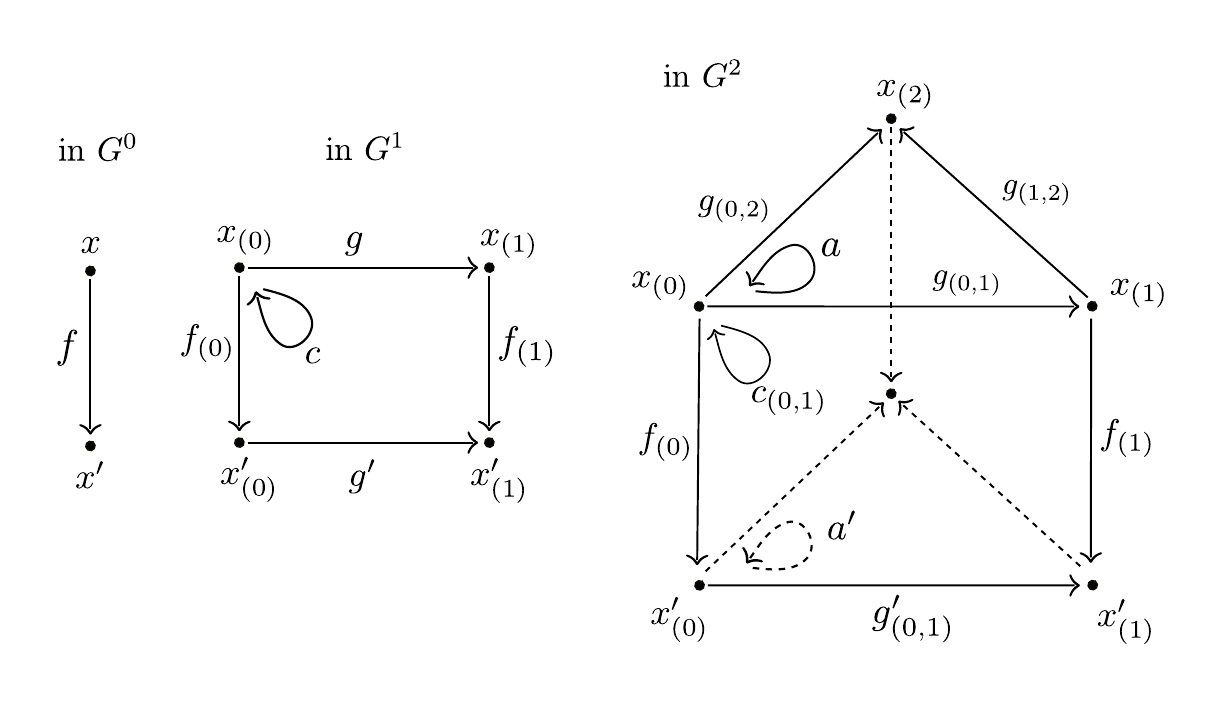}
\caption{Illustration of a gauge transformation 
$(f, c) : (x, g, a) \to (x', g', a')$
between descent data.}
\label{fig:3}
\end{center}
\end{figure}

\begin{prop}
Let $G$ be a cosimplicial crossed groupoid. The gauge 
transformations form an equivalence relation on the set
$\opn{Desc}(G)$.
\end{prop}

We call this relation {\em gauge equivalence}.
Actually $\opn{Desc}(G)$ has a lot more structure; see Remark \ref{rem:cosim.1}.

\begin{proof}
Given a descent datum $(x, g, a)$, the pair 
$(1_x, 1_{x_{(0)}})$, where $1_x \in G^0_1(x)$ and 
$1_{x_{(0)}} \in G^1_2(x_{(0)})$ are the identity elements of these groups,
is a gauge transformation from 
$(x, g, a)$ to itself. Next let
\[ (f, c) : (x, g, a) \to (x', g', a') \]
and
\[ (f', c') : (x', g', a') \to (x'', g'', a'') \]
be gauge transformations between descent data. Then 
\[ \bigl( f' \circ f \, , \, c \circ f_{(0)}^{-1}(c') \bigr) : (x, g, a) \to
(x'', g'', a'') \]
is a gauge transformation. And 
\[ \bigl( f^{-1}, f_{(0)}(c^{-1}) \bigr) : (x', g', a') \to (x, g, a) \]
is a gauge transformation. 
\end{proof}

Let $\Phi : G \to H$ be a morphism of cosimplicial
crossed groupoids. Given a descent datum 
$(x, g, a) \in \opn{Desc}(G)$, the triple 
\[ \Phi(x, g, a) := (\Phi(x), \Phi(g), \Phi(a)) \] 
is a descent datum in $H$. The resulting function 
\[ \opn{Desc}(\Phi) : \opn{Desc}(G) \to \opn{Desc}(H)  \]
respects the gauge equivalence relations. 

\begin{dfn} \label{dfn:cosim.1}
For a cosimplicial crossed groupoid $G$ we write
\[ \ol{\opn{Desc}}(G) :=
\frac{ \opn{Desc}(G) } {\tup{ gauge equivalence}} . \]
For a morphism $\Phi : G \to H$
of cosimplicial crossed groupoids, we denote by 
\[ \ol{\opn{Desc}}(\Phi) : \ol{\opn{Desc}}(G) \to
\ol{\opn{Desc}}(H) \]
the induced function. 
\end{dfn}

\section{The Main Theorem}
\label{sec:equiv-cosim}

Recall that for a groupoid $G$, the set of isomorphism classes of objects
is denoted by $\bsym{\bsym{\pi}}_0(G)$.

\begin{dfn} \label{dfn:equiv-cosim.101}
Let 
$G = 
( G_1, G_2, 
\opn{Ad}_{G_1 \crvar G_{2}}, \opn{D} )$
be a crossed groupoid. 
We define the homotopy set
\[ \bsym{\bsym{\pi}}_0(G) := \bsym{\bsym{\pi}}_0(G_1) , \]
and the homotopy groups
\[ \bsym{\bsym{\pi}}_1(G, x) := 
\opn{Coker} \bigl( \opn{D} : G_2(x) \to G_1(x) \bigr) \]
and
\[ \bsym{\bsym{\pi}}_2(G, x) := 
\opn{Ker} \bigl( \opn{D} : G_2(x) \to G_1(x) \bigr) \]
for $x \in \opn{Ob}(G)$.
\end{dfn}

The set $\bsym{\bsym{\pi}}_0(G)$ and the groups $\bsym{\bsym{\pi}}_i(G, x)$ are
functorial in $G$. 
The group $\bsym{\bsym{\pi}}_2(G, x)$ is central in $G_2(x)$, 
and in particular it is abelian.

\begin{dfn} \label{dfn:equiv-cosim.102}
A morphism of crossed groupoids $\Phi : G \to H$  
is called a {\em weak equivalence} if the function
\[ \bsym{\bsym{\pi}}_0(\Phi)  : \bsym{\bsym{\pi}}_0(G) \to
\bsym{\bsym{\pi}}_0(H)  \]
is bijective, and the group homomorphisms
\[ \bsym{\bsym{\pi}}_i(\Phi, x) : \bsym{\bsym{\pi}}_i(G, x) \to
\bsym{\bsym{\pi}}_i(H, \Phi(x))  \]
are bijective for all $x \in \opn{Ob}(G)$
and $i \in \{ 1, 2 \}$.
\end{dfn}

\begin{dfn} \label{dfn:equiv-cosim.103}
A morphism $\Phi : G \to H$
of cosimplicial crossed groupoids is called a {\em weak equivalence} if in
every simplicial dimension $p$ the morphism of crossed groupoids
 $\Phi^p : G^p \to H^p$ is a weak equivalence.
\end{dfn}

\begin{thm}[Equivalence] \label{thm:100} 
Let $\Phi : G \to H$ be a weak equivalence between 
cosimplicial crossed groupoids. Then the function 
\[ \ol{\opn{Desc}}(\Phi) : \ol{\opn{Desc}}(G) \to
\ol{\opn{Desc}}(H) \]
from Definition \tup{\ref{dfn:cosim.1}} is bijective.
\end{thm}

We need a couple of auxiliary results first. 
A partial descent datum in $G$ is a pair $(x, g)$ of elements
$x \in \opn{Ob}(G^0)$ and 
$g \in G^1_1(x_{(0)}, x_{(1)})$ (cf.\ Definition \ref{dfn:101}).
Let $(x, g)$ and $(x', g')$ be partial descent data. A partial gauge
transformation $(x, g) \to (x', g')$ is a pair $(f, c)$ of elements 
as in Definition \tup{\ref{dfn:100}}, that satisfies condition \tup{(i)} of that
definition.

The next lemma is a sort of ``Kan condition'' satisfied by 
$\opn{Desc}(G)$.

\begin{lem} \label{lem:104}
Let $(x, g, a)$ be a descent datum in the cosimplicial crossed groupoid
$G$, let $(x', g')$ be a partial
descent datum in $G$, and let $(f, c)$ be a partial gauge
transformation $(x, g) \to (x', g')$.  Then there is a unique element 
$a' \in G^2_2(x'_{(0)})$ such that the triple 
$(x', g', a')$ is a descent datum in $G$, and 
$(f, c)$ is a gauge transformation
$(x, g, a) \to (x', g', a')$.
\end{lem}

\begin{proof}
Define 
\[ a' := \opn{Ad}(f_{(0)}) \Bigl( 
c_{(0, 2)}^{-1} \circ a \circ \opn{Ad}(g_{(0, 1)}^{-1})(c_{(1, 2)})
\circ c_{(0, 1)}  \Bigr) \in G^2_2(x'_{(0)}) . \]
Then $a'$ satisfies condition (ii) of Definition \ref{dfn:100}, and moreover it
is unique. 

We have to show that the triple $(x', g', a')$ is a descent
datum. Let us check condition (i) of Definition \ref{dfn:101}. We have 
\[ \begin{aligned}
& (g'_{(0, 2)})^{-1} \circ g'_{(1, 2)} \circ g'_{(0, 1)}
\\
& \qquad  \overset{\vartriangle}{=} \Bigl( f_{(2)} \circ g_{(0, 2)} \circ 
\opn{D}(c_{(0, 2)}) \circ f_{(0)}^{-1} \Bigr)^{-1} 
\circ \Bigl( f_{(2)} \circ g_{(1, 2)} \circ  \opn{D}(c_{(1, 2)}) 
\circ f_{(1)}^{-1} \Bigr)
\\
& \qquad \qquad \circ \Bigl( f_{(1)} \circ g_{(0, 1)} \circ  
\opn{D}(c_{(0, 1)}) \circ f_{(0)}^{-1} \Bigr)
\\
& \qquad \overset{\lozenge}{=} f_{(0)} \circ \opn{D}(c_{(0, 2)})^{-1}
\circ  
g_{(0, 2)}^{-1} \circ g_{(1, 2)} \circ \opn{D}(c_{(1, 2)})
\circ g_{(0, 1)} \circ \opn{D}(c_{(0, 1)}) \circ f_{(0)}^{-1}
\\
& \qquad \overset{\heartsuit}{=} \opn{Ad}(f_{(0)}) \Bigl( 
\opn{D}(c_{(0, 2)}^{-1}) \circ 
\opn{D}(a) \circ g_{(0, 1)}^{-1} \circ \opn{D}(c_{(1, 2)})
\circ g_{(0, 1)} \circ \opn{D}(c_{(0, 1)}) \Bigr) 
\\
& \qquad \overset{\star}{=} \opn{Ad}(f_{(0)}) \Bigl( 
\opn{D}(c_{(0, 2)}^{-1}) \circ 
\opn{D}(a) \circ \opn{Ad}(g_{(0, 1)}^{-1})(\opn{D}(c_{(1, 2)})) 
\circ \opn{D}(c_{(0, 1)}) \Bigr) 
\\
& \qquad \overset{\square}{=} \opn{D} \Bigl( \opn{Ad}(f_{(0)}) \Bigl( 
c_{(0, 2)}^{-1} \circ a \circ \opn{Ad}(g_{(0, 1)}^{-1})(c_{(1, 2)})
\circ c_{(0, 1)} \Bigr) \Bigr) 
\overset{\triangledown}{=} \opn{D}(a') . 
\end{aligned} \]
The equality marked $\overset{\vartriangle}{=}$ is true because of  
condition (i) of Definition \ref{dfn:100}, applied to the elements 
$g'_{(i, j)}$. The equality marked $\overset{\lozenge}{=}$ is true because of
cancellation. The equality marked $\overset{\heartsuit}{=}$ is because 
condition (i) of Definition \ref{dfn:101} holds for $(x, g, a)$, and by the
definition of $\opn{Ad}(f_{(0)})$. 
The equality marked $\overset{\star}{=}$ is by the
definition of $\opn{Ad}(g_{(0,1)}^{-1})$. 
The equality marked $\overset{\square}{=}$ is because $\opn{D}$ is 
$G^2_1$-equivariant
(this is condition (i) of Definition \ref{dfn:cosim.101}).
And the equality marked $\overset{\triangledown}{=}$ holds by definition of 
$a'$.

Finally we have to check that condition (ii) of Definition \ref{dfn:101} holds
for $(x', g', a')$. Namely, letting 
\begin{equation}
u' := 
(a'_{(0, 1, 3)})^{-1} \circ a'_{(0, 2, 3)} \circ a'_{(0, 1, 2)} \circ
\opn{Ad} \bigl( (g'_{(0, 1)})^{-1} \bigr) (a'_{(1, 2, 3)})^{-1} ,
\end{equation}
we have to show that $u' = 1$.

{}From the definition of $a'$ we get
\begin{equation} \label{eqn:102}
\begin{aligned}
& \opn{Ad} \bigl( (g'_{(0, 1)})^{-1} \bigr) (a'_{(1, 2, 3)}) 
\\ 
& \qquad = \opn{Ad} \bigl( (g'_{(0, 1)})^{-1} \circ f_{(1)} \bigr)
\Bigl( c_{(1, 3)}^{-1} \circ a_{(1, 2, 3)} \circ
\opn{Ad}(g_{(1, 2)}^{-1})(c_{(2, 3)}) \circ c_{(1, 2)} \Bigr)
\\
& \qquad \overset{\heartsuit}{=} \opn{Ad} \bigl( f_{(0)} 
\circ \opn{D}(c_{(0, 1)}^{-1}) \circ (g_{(0, 1)})^{-1} \bigr)
\\
& \qquad  \qquad
\Bigl( c_{(1, 3)}^{-1} \circ a_{(1, 2, 3)} \circ
\opn{Ad}(g_{(1, 2)}^{-1})(c_{(2, 3)}) \circ c_{(1, 2)}  \Bigr)
\\
& \qquad \overset{\lozenge}{=}  
\bigl( \opn{Ad} (f_{(0)}) \circ \opn{Ad}(\opn{D}(c_{(0, 1)}^{-1})) 
\circ \opn{Ad}(g_{(0, 1)})^{-1} \bigr) 
\\
& \qquad  \qquad
\Bigl( c_{(1, 3)}^{-1} \circ a_{(1, 2, 3)} \circ
\opn{Ad}(g_{(1, 2)}^{-1})(c_{(2, 3)}) \circ c_{(1, 2)} \Bigr)
\\
& \qquad \overset{\square}{=} \opn{Ad} \bigl( f_{(0)} \bigr)
\Bigl( c_{(0, 1)}^{-1} \circ 
\opn{Ad} \bigl( g_{(0, 1)}^{-1} \bigr)(c_{(1, 3)}^{-1}) \circ
\opn{Ad} \bigl( g_{(0, 1)}^{-1} \bigr)(a_{(1, 2, 3)}) 
\\
& \qquad \qquad \circ
\opn{Ad} \bigl( g_{(0, 1)}^{-1} \circ g_{(1, 2)}^{-1} \bigr)(c_{(2, 3)}) \circ
\opn{Ad} \bigl( g_{(0, 1)}^{-1} \bigr)(c_{(1, 2)}) 
\circ c_{(0, 1)} \Bigr) .
\end{aligned}
\end{equation}
The equality marked $\overset{\heartsuit}{=}$ is true because
\[  (g'_{(0, 1)})^{-1} \circ f_{(1)} = f_{(0)} \circ \opn{D}(c_{(0, 1)}^{-1})
\circ (g_{(0, 1)})^{-1} ; \]
this is from condition (i) of Definition \ref{dfn:100}.
The equality marked $\overset{\lozenge}{=}$ is because $\opn{Ad}$ is a group
homomorphism. And $\overset{\square}{=}$ is because 
$\opn{Ad}(\opn{D}(c_{(0, 1)}^{-1})) = \opn{Ad}(c_{(0, 1)}^{-1})$, 
which is an instance of condition (ii) of Definition \ref{dfn:cosim.101}.

A consequence of condition (ii) of Definition \ref{dfn:101} and 
condition (ii) of Definition \ref{dfn:cosim.101} is that 
\[ \begin{aligned}
& a_{(0, 1, 2)}^{-1} \circ c \circ a_{(0, 1, 2)} =
\opn{Ad}(a_{(0, 1, 2)}^{-1})(c) 
\\
& \qquad = \opn{Ad}(\opn{D}(a_{(0, 1, 2)}^{-1}))(c) = 
\opn{Ad}(g_{(0, 1)}^{-1} \circ g_{(1, 2)}^{-1} \circ g_{(0, 2)})(c)
\end{aligned} \]
for any $c \in G^2_2(x_{(0)})$. 
Therefore, taking 
$c := \opn{Ad}(g_{(0, 2)}^{-1})(c_{(2, 3)})$, we get
\begin{equation} \label{eqn:103}
\opn{Ad}(g_{(0, 2)}^{-1})(c_{(2, 3)}) \circ a_{(0, 1, 2)}
 = a_{(0, 1, 2)} \circ 
\opn{Ad}(g_{(0, 1)}^{-1} \circ g_{(1, 2)}^{-1})(c_{(2, 3)}) .
\end{equation}

By the definition of $a'$ and by formula (\ref{eqn:102}) we have
\[ \begin{aligned}
& u' = 
\opn{Ad} ( f_{(0)} )
\Bigl(  c_{(0, 3)}^{-1} \circ a_{(0, 1, 3)}
\circ \opn{Ad}(g_{(0, 1)}^{-1})(c_{(1, 3)}) \circ c_{(0, 1)} \Bigr)^{-1}
\\
& \qquad \qquad \circ
\opn{Ad} ( f_{(0)} )
\Bigl( c_{(0, 3)}^{-1}  \circ a_{(0, 2, 3)} \circ
\opn{Ad}(g_{(0, 2)}^{-1})(c_{(2, 3)}) \circ c_{(0, 2)} \Bigr)
\\
& \qquad \qquad \circ  
\opn{Ad} ( f_{(0)} )
\Bigl(  c_{(0, 2)}^{-1} \circ a_{(0, 1, 2)} \circ
\opn{Ad}(g_{(0, 1)}^{-1})(c_{(1, 2)})  \circ c_{(0, 1)} \Bigr)
\\
& \qquad \qquad \circ \opn{Ad} \bigl( f_{(0)} \bigr)
\Bigl( c_{(0, 1)}^{-1} \circ 
\opn{Ad} \bigl( g_{(0, 1)}^{-1} \bigr)(c_{(1, 3)}^{-1}) \circ
\opn{Ad} \bigl( g_{(0, 1)}^{-1} \bigr)(a_{(1, 2, 3)}) 
\\
& \qquad \qquad  \qquad \circ 
\opn{Ad} \bigl( g_{(0, 1)}^{-1} \circ g_{(1, 2)}^{-1} \bigr)(c_{(2, 3)}) \circ
\opn{Ad} \bigl( g_{(0, 1)}^{-1} \bigr)(c_{(1, 2)}) 
\circ c_{(0, 1)} \Bigr)^{-1} \ .
\end{aligned} \]
Canceling adjacent inverse terms we get
\[ \begin{aligned}
& u' = 
\opn{Ad} ( f_{(0)} ) \Bigl(  c_{(0, 1)}^{-1} \circ 
\opn{Ad}(g_{(0, 1)}^{-1})(c_{(1, 3)}^{-1}) \circ 
v' \circ
\opn{Ad} \bigl( g_{(0, 1)}^{-1} \bigr)(c_{(1, 3)})
\circ c_{(0, 1)} \Bigr) \ , 
\end{aligned} \]
where
\[ \begin{aligned}
& v' := 
a_{(0, 1, 3)}^{-1} \circ a_{(0, 2, 3)} \circ 
\opn{Ad}(g_{(0, 1)}^{-1} \circ g_{(1, 2)}^{-1})(c_{(2, 3)})
\circ a_{(0, 1, 2)} 
\\
& \qquad \qquad \circ
\opn{Ad} \bigl( g_{(0, 1)}^{-1} \circ g_{(1, 2)}^{-1} \bigr)(c_{(2, 3)}^{-1})
\circ \opn{Ad} \bigl( g_{(0, 1)}^{-1} \bigr)(a_{(1, 2, 3)}^{-1}) \ .
\end{aligned} \]
It suffices to prove that $v' = 1$. 
Using formula (\ref{eqn:103}) we have 
\[ \begin{aligned}
& v' =  a_{(0, 1, 3)}^{-1} \circ a_{(0, 2, 3)}
\circ a_{(0, 1, 2)} \circ 
\opn{Ad}(g_{(0, 1)}^{-1} \circ g_{(1, 2)}^{-1})(c_{(2, 3)})
\\
& \qquad \qquad \circ
\opn{Ad} \bigl( g_{(0, 1)}^{-1} \circ g_{(1, 2)}^{-1} \bigr)(c_{(2, 3)}^{-1})
\circ \opn{Ad} \bigl( g_{(0, 1)}^{-1} \bigr)(a_{(1, 2, 3)}^{-1}) \ .
\end{aligned} \]
We now cancel two adjacent inverse terms, and use the fact that condition
(ii) of Definition \ref{dfn:101} holds for $(x, g, a)$, to conclude that 
$v' = 1$.
\end{proof}

Suppose $G$ is a crossed groupoid and $x, x' \in \opn{Ob}(G)$.
There is a right action of the group $G_2(x)$ on the set 
$G_1(x, x')$, namely 
$g \mapsto g \circ \opn{D}(a)$ for $g \in G_1(x, x')$
and $a \in G_2(x)$. The quotient set is 
\begin{equation}
\bsym{\bsym{\pi}}_1(G, x, x') := G_1(x, x') / G_2(x) . 
\end{equation}
Given $g, g' \in G_1(x, x')$ let us define 
\begin{equation}
G_2(x)(g, g') := \{ a \in G_2(x) \mid 
g' = g \circ \opn{D}(a) \} . 
\end{equation}
So $\bsym{\bsym{\pi}}_1(G, x, x) = \bsym{\bsym{\pi}}_1(G, x)$ and 
$G_2(x)(1_x, 1_x) = \bsym{\bsym{\pi}}_2(G, x)$
in the notation of Definition \ref{dfn:equiv-cosim.101}.

\begin{lem} \label{lem:cosim-equiv.100}
Let $\Phi : G \to H$ be a weak equivalence between crossed
groupoids. Then the induced functions 
\[ \bsym{\bsym{\pi}}_1(\Phi, x, x') : \bsym{\bsym{\pi}}_1(G, x, x') \to
\bsym{\bsym{\pi}}_1 \bigl( H, \Phi(x), \Phi(x') \bigr)  \]
and 
\[ \Phi : G_2(x)(g, g') \to 
H_2 \bigl( \Phi(x) \bigr) \bigl( \Phi(g), \Phi(g') \bigr) \]
are bijective for all $x, x' \in \opn{Ob}(G)$ and 
$f, f' \in G_1(x, x')$. 
\end{lem}

\begin{proof}
This is the same as the usual proof for $2$-groupoids (cf.\ 
\cite[Lemma 1.1]{MS}).
\end{proof}

\begin{proof}[Proof of Theorem \tup{\ref{thm:100}}]
The proof is a ``nonabelian diagram chasing'', made possible by
Lemma \ref{lem:104}.

We begin by proving that the function $\ol{\opn{Desc}}(\Phi)$ is
surjective. Given a descent datum $(y, h, b) \in \opn{Desc}(H)$, we
have to find a descent datum $(x, g, a) \in \opn{Desc}(G)$, and a gauge
transformation $(f, c) :  (y, h, b) \to \Phi(x, g, a)$ in $H$. 

Since the function 
$\bsym{\bsym{\pi}}_0(\Phi^0) : \bsym{\bsym{\pi}}_0(G^0) \to
\bsym{\bsym{\pi}}_0(H^0)$
is surjective, there is an object $x \in \opn{Ob}(G^0)$, and a
$1$-morphism $f \in H^0_1(y, y')$, where 
$y' := \Phi(x) \in \opn{Ob}(H^0)$. 
Define 
\[ h'' := f_{(1)} \circ h \circ f_{(0)}^{-1} \in 
H^0_1(y'_{(0)}, y'_{(1)})  \]
and 
$c'' := 1_{y_{(0)}} \in H^1_2(y_{(0)})$. 
Then $(y', h'')$ is a partial descent datum in $H$, and 
$(f, c''): (y, h) \to (y', h'')$ is a partial gauge transformation.
According to Lemma \ref{lem:104} there is a unique element 
$b'' \in H^2_2(y'_{(0)})$ such that 
$(y', h'', b'')$ is a descent datum in $H$, and 
$(f, c''): (y, h, b) \to (y', h'', b'')$ is a gauge transformation.

Now by Lemma \ref{lem:cosim-equiv.100} the function 
\[ \bsym{\bsym{\pi}}_1(\Phi^1, x_{(0)}, x_{(1)}) : 
\bsym{\bsym{\pi}}_1(G^1, x_{(0)}, x_{(1)}) \to
\bsym{\bsym{\pi}}_1 ( H^1, y'_{(0)}, y'_{(1)}) \]
is surjective. Hence there are elements
$g \in G^1_1(x_{(0)}, x_{(1)})$ and 
$c' \in H^1_2(y'_{(0)})$
such that, letting 
$h' := \Phi(g) \in H^1_1(y'_{(0)}, y'_{(1)})$, 
we have 
$h'' = h' \circ \opn{D}(c')$.
Consider the partial gauge transformation 
$(1_{y'}, c') : (y', h'') \to (y', h')$.
Lemma \ref{lem:104} there is a unique element 
$b' \in H^2_2(y'_{(0)})$ such that 
$(y', h', b')$ is a descent datum in $H$, and 
$(1_{y'}, c') : (y', h'', b'') \to (y', h', b')$
is a gauge transformation. Let 
$c := \opn{Ad}(f_{(0)}^{-1})(c')^{-1} \in H^1_2(y_{(0)})$.
Then 
\[ (f, c) :  (y, h, b) \to (y', h', b') \]
is a gauge transformation in $H$, and $(y', h') = \Phi(x, g)$. 

By Lemma \ref{lem:cosim-equiv.100} the function
\[ \begin{aligned}
& \Phi^2 : G^2_2 (x_{(0)})
\bigl( 1_{x_{(0)}}, g_{(0, 2)}^{-1} \circ g_{(1, 2)} \circ g_{(0, 1)}
\bigr) \\
& \qquad \to 
H^2_2 (y'_{(0)})
\bigl( 1_{y'_{(0)}}, (h'_{(0, 2)})^{-1} \circ h'_{(1, 2)} 
\circ h'_{(0, 1)} \bigr)
\end{aligned} \]
is bijective. Let $a \in G^2_2 (x_{(0)})$
be the unique element such that 
\[ \opn{D}(a) = g_{(0, 2)}^{-1} \circ g_{(1, 2)} \circ g_{(0, 1)} \]
and $\Phi(a) = b'$. Then the triple of elements 
$(x, g, a)$ satisfies condition (i) of Definition \ref{dfn:101}, and 
$\Phi(x, g, a) = (y', h', b')$. 
Now the element 
\[ u := a_{(0, 1, 3)}^{-1} \circ a_{(0, 2, 3)} \circ 
a_{(0, 1, 2)} \circ \opn{Ad}(g_{(0, 1)}^{-1})(a_{(1, 2, 3)})^{-1} 
\in G^3_2(x_{(0)}) \]
satisfies $\opn{D}(u) = 1$, so it belongs to the subgroup 
$\bsym{\bsym{\pi}}_2(G^3, x_{(0)}) \subset G^3_2(x_{(0)})$.
Since the group homomorphism 
\[ \bsym{\bsym{\pi}}_2(\Phi^3, x_{(0)}) : \bsym{\bsym{\pi}}_2(G^3, x_{(0)}) \to 
\bsym{\pi}_2(H^3, y'_{(0)}) \]
is injective, and since 
\[ \begin{aligned}
& \bsym{\pi}_2(\Phi^3, x_{(0)})(u) = \Phi^3(u) =
\\
& \qquad  
(b'_{(0, 1, 3)})^{-1} \circ b'_{(0, 2, 3)} \circ 
b'_{(0, 1, 2)} \circ \opn{Ad}(h'_{(0, 1)})^{-1}(b'_{(1, 2, 3)})^{-1} = 1 ,
\end{aligned} \]
we conclude that $u = 1$. Thus the triple $(x, g, a)$ satisfies 
condition (ii) of Definition \ref{dfn:101}, so it is a descent datum in 
$G$. 

Now we prove that the function $\ol{\opn{Desc}}(\Phi)$ is
injective. Given 
\[ (x, g, a), (x', g', a') \in \opn{Desc}(G) , \]
define
$(y, h, b) := \Phi(x, g, a)$ and $(y', h', b') := \Phi(x', g', a')$. Assume we
are given a gauge transformation 
\[ (f, c) :  (y, h, b) \to (y', h', b') \]
in $H$. We have to produce a gauge transformation 
\[ (e, d) :  (x, g, a) \to (x', g', a') \]
in $G$.

We know that the function 
\[ \bsym{\pi}_1(\Phi^0, x, x') : 
\bsym{\pi}_1(G^0, x, x') \to
\bsym{\pi}_1 \bigl( H^0, y, y' \bigr) \]
is surjective. Therefore there is a $1$-morphism 
$e \in G^0_1(x, x')$, and a $2$-morphism 
$v \in H^0_2(y)$, 
such that 
$\Phi(e) = f \circ \opn{D}(v)$. 
Let 
\[ \til{f} := f \circ \opn{D}(v) \in H^0_1(y, y') \]
and
\[ \til{c} := \opn{Ad}(h^{-1})(v_{(1)}^{-1}) \circ c \circ v_{(0)} \in 
H^0_2(y_{(0)}) . \]
A simple calculation shows that 
\[ (\til{f}, \til{c}) :  (y, h, b) \to (y', h', b') \]
is also a gauge transformation. 

Recall that $\Phi(e) = \til{f}$, so  
\[ \Phi(g^{-1} \circ e_{(1)}^{-1} \circ g' \circ e_{(0)}) = 
h^{-1} \circ \til{f}_{(1)}^{-1} \circ h' \circ \til{f}_{(0)} =
\opn{D}(\til{c}) . \]
This is condition (i) of Definition \ref{dfn:100} for the gauge transformation 
$(\til{f}, \til{c})$.
Because the group homomorphism 
\[  \bsym{\pi}_1(\Phi^1, x_{(0)}) : 
\bsym{\pi}_1(G^1, x_{(0)}) \to
\bsym{\pi}_1 ( H^1, y_{(0)}) \]
is injective, it follows that there is an element 
$d' \in G^1_2(x_{(0)})$ such that 
\[ g^{-1} \circ e_{(1)}^{-1} \circ g' \circ e_{(0)} = \opn{D}(d') . \]
Consider the element 
$w := \til{c} \circ \Phi(d')^{-1} \in H^1_2(y_{(0)})$. 
It satisfies $\opn{D}(w) = 1$, so it belongs to the subgroup 
$\bsym{\pi}_2(H^1, y_{(0)})$. But the homomorphism 
\[ \bsym{\pi}_2(\Phi^1, x_{(0)}) : 
\bsym{\pi}_2(G^1, x_{(0)}) \to
\bsym{\pi}_2 ( H^1, y_{(0)}) \]
is bijective, so there is a unique element 
$v \in G^1_2(x_{(0)})$ satisfying $\opn{D}(v) = 1$ and 
$\Phi(v) = w$.
Let $d := v \circ d' \in G^1_2(x_{(0)})$.
Then $\Phi(d) = \til{c}$ and
\[ \opn{D}(d) = \opn{D}(d') = g^{-1} \circ e_{(1)}^{-1} \circ g' \circ e_{(0)}
. \]
Thus the pair $(e, d)$ is a partial gauge transformation 
$(x, g) \to (x', g')$. 

The last thing to check is that condition (ii) of Definition \ref{dfn:100}
holds for $(e, d)$. Let 
\[ u := \opn{Ad}(e_{(0)}^{-1})(a')^{-1} \circ    
d_{(0, 2)}^{-1} \circ a \circ 
\opn{Ad}(g_{(0, 1)}^{-1})(d_{(1, 2)}) \circ d_{(0, 1)}
\in  G^2_2(x_{(0)}) . \]
A direct calculation shows that 
$\opn{D}(u) = 1_{x_{(0)}} \in G^2_1(x_{(0)})$, so 
$u \in \bsym{\pi}_2(G^2, x_{(0)})$. 
We know that $\Phi(u) = 1_{y_{(0)}}$, and that the homomorphism 
\[ \bsym{\pi}_2(\Phi^2 ,x_{(0)}) : \bsym{\pi}_2(G^2 ,x_{(0)}) \to
\bsym{\pi}_2(H^2 ,y_{(0)}) \] 
is injective. It follows that $u = 1_{x_{(0)}}$, which is
what we had to check.
\end{proof}

\section{Some Remarks}
\label{sec:remarks}

To finish the note, here are a few remarks and clarifications, to help place
our work in context. 

\begin{rem} \label{rem:10}
Cosimplicial crossed groupoids arise naturally in the geometry of gerbes. 
This is well-known -- see \cite{Gi, Br1, Br2, BGNT1}. 
Let us quickly indicate how this occurs. Let $\mcal{G}$ be a sheaf of groups on
a topological space $X$, and let $\mcal{A}ut(\mcal{G})$ be its sheaf of
automorphism groups. There is an obvious sheaf of crossed
groups (i.e.\ crossed modules)
\[ \gerbe{G} = \big( \opn{D} : \mcal{G} \to \mcal{A}ut(\mcal{G}) \big)  \] 
on $X$. Consider an open covering $\bsym{U} = \{ U_k \}$ of $X$. 
The \v{C}ech construction gives rise to a cosimplicial crossed group
$G := \mrm{C}(\bsym{U} , \gerbe{G})$.
The set $\ol{\opn{Desc}}(G)$ is an approximation of {\em the set of
equivalence classes of $\mcal{G}$-gerbes} on $X$ (the discrepancy is because we
may have to use refinements and hypercoverings). 

If  every $\mcal{G}$-gerbe totally trivializes on the covering $\bsym{U}$ (see
\cite[Definition 9.16]{Ye1}),
then $\ol{\opn{Desc}}(G)$ actually classifies $\mcal{G}$-gerbes.
If $\bsym{V}$ is another such covering, and 
$\bsym{V} \to \bsym{U}$ is a refinement, then there is an induced 
weak equivalence of cosimplicial crossed groups
$\mrm{C}(\bsym{U} , \gerbe{G}) \to \mrm{C}(\bsym{V} , \gerbe{G})$.

More general gerbes (not $\mcal{G}$-gerbes as above) 
can sometimes be classified by a {\em sheaf of crossed groupoids} 
$\gerbe{G}$ -- see next remark.
For the full generality is it necessary to invoke more complicated
combinatorics -- see Remark \ref{rem:15}.
\end{rem}

\begin{rem} \label{rem:12}
Theorem \ref{thm:1} is used to prove {\em twisted deformation quantization} in
our paper \cite{Ye1} (see also the survey article \cite{Ye3}). We look at a
smooth algebraic variety $X$ over a field $\K$ of characteristic $0$.
By parameter $\K$-algebra we mean a complete noetherian local commutative 
$\K$-algebra $R$, with maximal ideal $\m$ and residue field 
$R / \m = \K$. The main example is $R = \K[[\hbar]]$, the algebra of
power series in a variable $\hbar$. 

On the variety $X$ there are two important sheaves of DG Lie algebras: the
algebra $\mcal{T}_{\mrm{poly}, X}$ of poly derivations, and the algebra 
$\mcal{D}_{\mrm{poly}, X}$ of poly differential operators. 
Consider the sheaves of pronilpotent DG Lie algebras
$\m \hot \mcal{T}_{\mrm{poly}, X}$ and 
$\m \hot \mcal{D}_{\mrm{poly}, X}$. The {\em Deligne crossed groupoid}
construction (see \cite{Ge} or \cite{Ye2}) gives rise to 
sheaves of crossed groupoids 
$\gerbe{G} := \opn{Del} \bigl( \m \hot \mcal{T}_{\mrm{poly}, X} \bigr)$
and
$\gerbe{H}:= \opn{Del} \bigl( \m \hot \mcal{D}_{\mrm{poly}, X} \bigr)$.

Now let $\bsym{U}$ be a finite affine open covering of $X$. 
Consider the cosimplicial crossed groupoids 
$G := \mrm{C}(\bsym{U} , \gerbe{G})$
and 
$H := \mrm{C}(\bsym{U} , \gerbe{H})$.
The sets $\ol{\opn{Desc}}(G)$ and $\ol{\opn{Desc}}(H)$
classify {\em twisted Poisson $R$-deformations} and {\em
twisted associative $R$-deformations} of $\mcal{O}_X$, respectively.
These twisted deformations are stacky versions of usual deformations 
(similar to the stacks of algebroids in \cite{Ko, KS}). The fact that twisted
$R$-deformations of $\mcal{O}_X$ totally trivialize on any affine open covering 
relies on our paper \cite{Ye4}.

There is a diagram of weak equivalences of cosimplicial crossed groupoids
\[ \UseTips \xymatrix @C=8ex @R=6ex {
G
\ar[d]_{\Phi_{\til{G}}}
&
H
\ar[d]^{\Phi_{\til{H}}}
\\
\til{G}
\ar[r]^{\til{\Psi}}
&
\til{H}
} \]
The weak equivalences $\Phi_{\til{G}} : G \to \til{G}$ and
$\Phi_{\til{H}} : H \to \til{H}$ come from suitable 
resolutions of the sheaves of DG Lie algebras 
$\mcal{T}_{\mrm{poly}, X}$ and $\mcal{D}_{\mrm{poly}, X}$ respectively. The 
weak equivalence $\til{\Psi} : \til{G} \to \til{H}$ comes from the Kontsevich 
Formality Theorem and its adaptation to the algebro-geometric setting; see 
\cite[Theorems 10.3 and 10.6]{Ye1}. Due to Theorem \ref{thm:1} we obtain a 
bijection  
\[ \opn{tw{.}quant} : \ol{\opn{Desc}}(G) \to \ol{\opn{Desc}}(H) , \]
\[ \opn{tw{.}quant} := \ol{\opn{Desc}}(\Phi_{\til{H}})^{-1} \circ 
\ol{\opn{Desc}}(\til{\Psi}) \circ \ol{\opn{Desc}}(\Phi_{\til{G}}) . \]
A refinement argument shows that $\opn{tw{.}quant}$ is independent of the 
resolutions. We call it the {\em twisted quantization map}. 
 
The twisted quantization map $\opn{tw{.}quant}$ does not arise from a
refinement or any similar geometric operation. Indeed, it is conjectured that in
some cases (e.g.\ when $X$ is a Calabi-Yau surface) the twisted quantization 
would send a sheaf to a stack, thus destroying the geometry. 

An earlier approach to twisted deformation quantization (in previous versions
of \cite{Ye1}) relied on {\em nonabelian multiplicative integration on
surfaces} \cite{Ye6}. The new approach uses Theorem \ref{thm:1}, and is
much shorter. Let us mention however that the ideas of \cite{Ye6} were recently
picked up by \cite{BGNT2}, for comparing various nerves of a $2$-groupoid. 
\end{rem}

\begin{rem} \label{rem:11}
A crossed groupoid (i.e.\ crossed module over a groupoid) is the same as a {\em
strict $2$-groupoid}; see \cite[Proposition 5.5]{Ye2}. The homotopy set
$\bsym{\pi}_0(G)$ and 
groups $\bsym{\pi}_i(G, x)$ are the same in both incarnations.
A crossed groupoid  is also the same as a 
category with inner gauge groups $(\cat{P}, \opn{IG}, \opn{ig})$ where
$\cat{P}$ is a groupoid -- see \cite[Section 5 and Proposition 10.4]{Ye1}. 

Traditionally papers used the $2$-groupoid language to discuss descent for
gerbes (cf.\ \cite{BGNT1}). In \cite{Ye1} we realized that the crossed groupoid
language is more effective: geometric descent data comes naturally in terms of a
cosimplicial crossed groupoid (see \cite[Section 10]{Ye1}), and also the Deligne
construction (see \cite[Section 6]{Ye2}) appears as a crossed groupoid, so it is
more natural to talk about the Deligne crossed groupoid. 

Another reason for preferring the crossed groupoid language is that defining
combinatorial descent data, and proving Theorem \ref{thm:100}, is easier this
way.
\end{rem}

\begin{rem}
To the best of our knowledge there is no prior proof of Theorem \ref{thm:1},
and moreover such a general result was not even conjectured before.

Results of similar flavor do appear in the work of Breen \cite{Br1, Br2} on 
classification of gerbes. Indeed Breen shows that the morphism of crossed
groups arising from a refinement (cf.\ Remark \ref{rem:10}) 
induces a bijection on $\ol{\opn{Desc}}(-)$. 
It is possible that Theorem \ref{thm:1} can be given a shorter proof using
Breen's diagram calculus. 

Another earlier result, a special case of our Theorem \ref{thm:1}, is
\cite[Proposition 3.3.1]{BGNT1}. However it is too restricted for the purposes
of \cite{Ye1}. 

Assertions in \cite{CH} that are similar to Theorem
\ref{thm:1} are, to the best of our understanding, not actually proved there. 

The paper \cite{Pr} by Prezma, which is a follow-up to our own present note, 
extends Theorem \ref{thm:1} significantly (see next remark). The proof given
in \cite{Pr} uses methods from homotopy
theory, and is much more sophisticated than our direct calculations. 
Nonetheless our proof is somewhat shorter, and is completely self-contained. 
\end{rem}

\begin{rem} \label{rem:cosim.1}
Let $G \in \bsym{\Delta} (\cat{CrGrpd})$.
The set $\opn{Desc}(G)$ has a crossed groupoid structure, in
which the gauge transformations are the $1$-morphisms; so 
$\bsym{\pi}_0(  \opn{Desc}(G) ) = \ol{\opn{Desc}}(G)$.
See \cite{Pr}, and also \cite[Section 5]{Br2}, \cite{Hi} and \cite{BGNT1} for
special cases. A morphism $\Phi : G \to H$ in 
$\bsym{\Delta} (\cat{CrGrpd})$ induces a morphism 
\[ \opn{Desc}(\Phi) : \opn{Desc}(G) \to 
\opn{Desc}(H) \]
in $\cat{CrGrpd}$. 

Prezma \cite{Pr} proves that if $\Phi$ is a cosimplicial weak equivalence,  then 
the morphism of crossed groupoids $\opn{Desc}(\Phi)$ is a {\em weak
equivalence}. This of course implies Theorem \ref{thm:1}.
The result of Prezma extends Jardine's
corresponding result for (ordinary) groupoids \cite{Ja}.
Moreover Prezma shows how to extend his equivalence theorem to higher 
($n > 2$) groupoids.
\end{rem}

\begin{rem} \label{rem:15}
We should emphasize that Theorem \ref{thm:1} holds in much greater
generality than stated. First observe that we made no use of the full
cosimplicial structure of the cosimplicial crossed groupoids $G$ and 
$H$ -- all we needed is that these are {\em $3$-truncated
semi-cosimplicial crossed groupoids}. Namely that $G$ and 
$H$ are functors 
$\bsym{\Delta}^{\leq 3}_{\mrm{inj}} \to \cat{CrGrpd}$, where
$\bsym{\Delta}^{\leq 3}_{\mrm{inj}}$ is the subcategory of 
$\bsym{\Delta}$ on the object set $\{ 0, 1, 2, 3 \}$, and with only injective 
functions $\al : p \to q$. Likewise for the weak equivalence 
$\Phi : G \to H$.

The reason  we chose to use the cosimplicial language is convenience. It
also fits nicely with the conventions of \cite{Ye1}. 

Other papers dealing with descent for gerbes and stacks (cf.\ \cite{Br1, Br2,
DP}) indicate that Theorem \ref{thm:1} might be extended further. It is almost
clear to us that the theorem holds for {\em $4$-truncated
pseudo-semi-cosimplicial crossed groupoids}, namely for pseudo-functors
$G, H : \bsym{\Delta}^{\leq 4}_{\mrm{inj}} \to \cat{CrGrpd}$.
This would correspond to a \v{C}ech construction performed on a {\em stack
of crossed groupoids} on a space $X$. 

A more complicated extension would be to replace the $2$-category 
$\cat{CrGrpd}$ of crossed groupoids with the $2$-category
$\cat{w2Grpd}$  of {\em weak $2$-groupoids}, i.e.\ to work with pseudo-functors
$\bsym{\Delta}^{\leq 4}_{\mrm{inj}} \to \cat{w2Grpd}$.

It is reasonable to assume that proving such extended results would require
either a bare-hands approach like ours, or else a very sophisticated
homotopical machinery.
\end{rem}


\end{document}